\documentclass[a4paper,11pt]{article}

\IfFileExists{styleart.sty}{\input{styleart.sty}}{%
  \usepackage{amsmath,amssymb,amsthm}
  \newtheorem{theorem}{Theorem}[section]
  \newtheorem{proposition}[theorem]{Proposition}
  \newtheorem{lemma}[theorem]{Lemma}
  \newtheorem{fact}[theorem]{Fact}
  \theoremstyle{definition}
  \newtheorem{definition}[theorem]{Definition}
}

\title{The theory of nonabelian free groups has no model companion}
\date{\today}

\author{O. Kharlampovich, R. Sklinos%
\thanks{The first-named author thanks the Dolciani--Halloran Foundation for its support.
The second-named author was supported by the National Natural Science Foundation of China (NSFC), grant no.~12350610234.}}

\begin{document}

\maketitle

\begin{abstract}
We prove that the first-order theory of nonabelian free groups does not
admit a model companion. 
\end{abstract}
\section{Introduction}
The  results of \cite{MR2211515} and  \cite{MR2238945} on the
completeness of the first-order theory of nonabelian free groups opened
the way to the systematic study of its model theory. A fundamental
question in this direction is whether this theory admits a model
companion. The question belongs to a classical line of model-theoretic
algebra going back to Robinson.

The classical example is the theory of fields, whose model companion is
the theory of algebraically closed fields. In the case of nonabelian
free groups, despite the understanding of their first-order
theory provided by the solution of Tarski's problem, the existence of a
model companion has remained open. We prove that no such theory exists.

\begin{theorem} \label{thm:main}
The first-order theory of nonabelian free groups does not admit a model
companion.
\end{theorem}

The proof rests on an obstruction coming from existential formulas and
proper powers in free groups. If a model companion existed, the set of
nonsquares would be existentially definable in it. Using uniform small
cancellation together with the implicit function theorem over free
groups, we show that such a formula must be satisfied by a primitive
element of a free group. But an existential formula which holds of a
primitive element also holds of each of its nontrivial powers. In
particular, the formula defining nonsquares would hold of a square,
yielding a contradiction.

\section{Preliminaries}

Throughout, a first-order theory is assumed to be consistent.

\begin{definition}[Existentially closed substructure]
Let $\mathcal A\subseteq\mathcal B$ be first-order structures in the
same language. We say that $\mathcal A$ is existentially closed in
$\mathcal B$ if for every
existential formula $\phi(\bar{x})$ and every tuple
$\bar a\in A$, if $\mathcal B\models\phi(\bar a)$, then $\mathcal A\models\phi(\bar a)$.
\end{definition}

\begin{definition}[Model complete theory]
A first-order theory $\mathcal{T}$ is model complete if for any two models $\mathcal{A}$, $\mathcal{B}$ of $\mathcal{T}$ with $\mathcal{A}\subseteq\mathcal{B}$ we have that $\mathcal{A}$ is existentially closed in $\mathcal{B}$.
\end{definition}

\begin{fact}(\cite[Theorem 3.5.1]{MR1347464})\label{ModelComplete}
The first-order theory $\mathcal{T}$ is model complete if and only if every first-order formula is equivalent modulo $\mathcal{T}$ to an existential formula.
\end{fact}

Let $\mathcal{T}$ be a first-order theory. We denote by $\mathcal{T}^{\forall}$ the set of all universal first-order sentences which are consequences of $\mathcal{T}$.

\begin{definition}[Model companion]
Let $\mathcal{T}$ be a first-order theory. Then $\mathcal{T}$ admits a model companion if there exists a model complete first-order theory $\mathcal{U}$ such that $\mathcal{T}^{\forall}=\mathcal{U}^{\forall}$. The theory $\mathcal{U}$ is called the model companion of $\mathcal{T}$.
\end{definition}

We now specialize to the first-order theory of nonabelian free groups.
It is immediate that this theory is not model complete. Indeed, consider
the subgroup
\[
\langle e_1,e_2^2\rangle \leq \mathbb{F}_2=\langle e_1,e_2\rangle.
\]
Although $\langle e_1,e_2^2\rangle$ is itself a free group of rank $2$,
the inclusion is not elementary: the element $e_2^2$ has a square root
in $\mathbb{F}_2$, whereas it has no square root in
$\langle e_1,e_2^2\rangle$.

We denote the first-order theory of nonabelian free groups by $T_{fg}$. We will use the following immediate consequence of the completeness of $T_{fg}$: if $\mathcal{U}^{\forall}=T_{fg}^{\forall}$ and an existential sentence holds in some model of $\mathcal{U}$, then it holds in $\mathbb F_2$ (in fact in any model of $T_{fg}$). 

Finally, we record a simple but important observation concerning
existential formulas in free groups. Namely, an existential formula
cannot distinguish a primitive element from any of its nontrivial
powers. More precisely, if $\phi(x)$ is an existential formula, $a$ is
a primitive element of $\mathbb{F}_n$, and
$\mathbb{F}_n \models \phi(a)$, then
$\mathbb{F}_n \models \phi(a^k)$ for every nonzero integer $k$.
Indeed, writing
\[
\mathbb{F}_n=\mathbb{F}_{n-1}*\langle a\rangle,
\]
it admits an embedding into
$\mathbb{F}_{n-1}*\langle b\rangle$ which is the identity on
$\mathbb{F}_{n-1}$ and sends $a$ to $b^k$. Since existential formulas
are preserved under embeddings, $\phi(b^k)$ holds in
$\mathbb{F}_{n-1}*\langle b\rangle$, which is isomorphic to
$\mathbb{F}_n$.

\section{Small cancellation and formal solutions}

Throughout, small cancellation is understood with respect to the action of
$\mathbb F_n$, $n>1$, on the Cayley tree associated with a fixed basis.
We denote by $tl(a)$ the translation length of $a\in\mathbb F_n$.

For our purposes, it suffices to consider the following specialization of
the $N$-small cancellation condition of
\cite[Definition 3.2]{MR4206600}.

\begin{definition}
Let $a$ be an element of a nonabelian free group $\mathbb F_n$ and
let $N$ be a positive integer. We say that $a$ is $N$-small cancellation
if there exists a vertex $*$ of the Cayley tree such that:
\begin{itemize}
\item $d(*,a.*)>N$;
\item $d(*,a.*)\leq \frac{N+1}{N}tl(a)$;
\item for every $g\in\mathbb F_n$, if
$\bigl|[*,a.*]\cap g[*,a. *]\bigr|
\geq \frac1N tl(a)$, then $g=1$.
\end{itemize}
\end{definition}

Small cancellation elements can be produced uniformly
from arbitrary pairs of noncommuting elements. The following consequence
of results in \cite{MR4206600} makes this precise.

\begin{proposition}\label{SmallCancellation}
For every positive integer $N$ there exists a family
\[
\mathcal W_N=\{w_1(p,q),w_2(p,q)\}\subseteq\mathbb F(p,q)
\]
such that, for every $n>1$ and every pair of noncommuting elements
$a,b\in\mathbb F_n$, at least one of the elements $w_1(a,b), w_2(a,b)$ is $N$-small cancellation.
\end{proposition}

\begin{proof}
Since $a$ and $b$ do not commute, neither do $a^{10}$ and $b^{10}$,
and hence $\langle a^{10},b^{10}\rangle$ acts irreducibly on the
Cayley tree. Moreover, every nontrivial element of $\mathbb F_n$ is
stable for this action. Thus Lemma~4.6 of \cite{MR4206600} implies that
$(a^{10},b^{10})$ is an acylindrical pair.

By Lemma~4.7 of \cite{MR4206600}, the element
\[
c:=w^{cn}(a^{10},b^{10})
\]
is hyperbolic, $a^{10}$ does not preserve the axis of $c$, and the axes
of $a^{10}$ and $c$ meet coherently. In particular,
\[
d(A(a^{10}),A(c))=0<tl(c).
\]
Moreover, $c$ is stable and, since $a^{10}$ does not preserve the axis
of $c$, the subgroup $\langle a^{10},c\rangle$ is noncyclic and hence
acts irreducibly on the Cayley tree. Corollary~4.10 of
\cite{MR4206600}, applied with $m=1$, therefore gives two words
$u_1(p,q),u_2(p,q)$ such that at least one of
\[
u_1(a^{10},c),\qquad u_2(a^{10},c)
\]
is $N$-small cancellation. Hence we may take
\[
\mathcal W_N=
\left\{
u_1\bigl(p^{10},w^{cn}(p^{10},q^{10})\bigr),
u_2\bigl(p^{10},w^{cn}(p^{10},q^{10})\bigr)
\right\}.
\]
\end{proof}

The following theorem is a version of the implicit function theorem
\cite[Theorem 4]{MR2154989} (see also
\cite[Theorem 1.2]{MR2060243}). Originating in Merzlyakov's work on
systems of equations \cite{Merzlyakov1966}, the theorem was subsequently
extended to incorporate inequalities and became a fundamental ingredient
in the proof of the completeness of the first-order theory of nonabelian
free groups. The particular formulation needed here is standard and
follows from the usual proof of the implicit function theorem.

\begin{theorem}\label{ImplicitlLimitGroups}
Let $\Sigma(x,\bar y)=1$ and $\Psi(x,\bar y)\neq 1$ be, respectively,
a finite system of equations and a finite conjunction of inequalities,
all coefficient-free. Then there exists a positive integer $N$,
depending only on $\Sigma$, with the following property.
Suppose that $n>1$, that $a\in\mathbb F_n$ is $N$-small
cancellation, and that
\[
\mathbb F_n\models
\exists\bar y\bigl(\Sigma(a,\bar y)=1\land
\Psi(a,\bar y)\neq 1\bigr).
\]
Then there exists a tuple of words $\bar w(x)$ with coefficients
in $\mathbb F_n$ such that every word appearing in
$\Sigma(x,\bar w(x))$ is trivial in
$\mathbb F_n*\mathbb F(x)$ and every word appearing in
$\Psi(x,\bar w(x))$ is nontrivial in
$\mathbb F_n*\mathbb F(x)$.

Consequently,
\[
\mathbb F_n*\mathbb F(z)\models
\exists\bar y\bigl(\Sigma(z,\bar y)=1\land
\Psi(z,\bar y)\neq 1\bigr).
\]
\end{theorem}



The preceding results yield the following finite test for coefficient-free
existential formulas.

\begin{lemma}\label{FiniteTest}
Let $\phi(x)$ be a coefficient-free existential formula. Then there exists
a finite nonempty family $\mathcal W_\phi\subseteq\mathbb F(p,q)$ such that, for every $n>1$ and every pair of noncommuting elements
$a,b\in\mathbb F_n$,
$$\mathbb F_n\models
\bigwedge_{w\in\mathcal W_\phi}\phi(w(a,b))
\quad\Longrightarrow\quad
\mathbb F_n*\langle z\rangle\models\phi(z).
$$
\end{lemma}
\begin{proof}
Write $\phi(x)$ as a finite disjunction
\[
\phi(x)=
\bigvee_{j=1}^r
\exists\bar y\,
\bigl(\Sigma_j(x,\bar y)=1\land
\Psi_j(x,\bar y)\neq1\bigr),
\]
where, by adjoining the equation $1=1$ if necessary, we may assume that
each $\Sigma_j$ is nonempty. Choose $N$ so that
Theorem~\ref{ImplicitlLimitGroups} applies to every disjunct, and let
$\mathcal W_\phi:=\mathcal W_N$ be given by
Proposition~\ref{SmallCancellation}.

Suppose that $\phi(w(a,b))$ holds for every
$w\in\mathcal W_\phi$. By Proposition~\ref{SmallCancellation}, some
$w(a,b)$ is $N$-small cancellation. Applying
Theorem~\ref{ImplicitlLimitGroups} to a disjunct of $\phi(w(a,b))$
which holds, we obtain
\[
\mathbb F_n*\langle z\rangle\models\phi(z),
\]
as required.
\end{proof}

\section{Proof of the main theorem}

\begin{proof}[Proof of Theorem~\ref{thm:main}]
Suppose that $\mathcal U$ is a model companion of $T_{fg}$. By Fact \ref{ModelComplete}, there is an existential formula $\phi(x)=\exists\bar y\,\psi(x,\bar y)$, with $\psi$ quantifier-free, such that
\begin{equation}\label{Nonsquares}
 \mathcal U\models\forall x\,
 \bigl(\phi(x)\leftrightarrow\neg\exists y\,(y^2=x)\bigr).
\end{equation}
Let
\[
 \theta(x):=\bigl(\exists y\,(y^2=x)\bigr)\vee\phi(x).
\]
This is an existential formula, and $\mathcal U\models\forall x\,\theta(x)$. Let $\mathcal W_\theta$ be the finite family supplied by Lemma \ref{FiniteTest}.

The theory $\mathcal U$ has a nonabelian model $G$. Otherwise the universal sentence $\forall p\forall q\,[p,q]=1$ would belong to $\mathcal U^{\forall}=T_{fg}^{\forall}$, a contradiction. Choose noncommuting $a,b\in G$. Then $G$ satisfies the sentence
\begin{equation}\label{FiniteSentence}
 \exists p\exists q\left([p,q]\neq1\ \wedge\
  \bigwedge_{w\in\mathcal W_\theta}\theta(w(p,q))\right).
\end{equation}
Since $\mathcal W_\theta$ is finite, this sentence is existential. By the completeness of $T_{fg}$ (see the observation at the end of Section~2) we get that \eqref{FiniteSentence} holds in $\mathbb F_2$. Applying Lemma \ref{FiniteTest} to its witnesses gives
\begin{equation}\label{FreeGenerator}
 \mathbb F_2*\langle z\rangle\models\theta(z).
\end{equation}

The element $z$ has no square root in $\mathbb F_2*\langle z\rangle$. Thus (3) implies
$\mathbb F_2*\langle z\rangle\models\phi(z)$. Since $z$ is primitive in the free group
$\mathbb F_2*\langle z\rangle$, the observation in the last paragraph of Section~2 gives
$
\mathbb F_2*\langle z\rangle\models\phi(z^2).
$
On the other hand, (1) implies the universal sentence
\[
\forall x\forall\bar y\forall t\,
\neg\bigl(\psi(x,\bar y)\land t^2=x\bigr),
\]
which belongs to $U^\forall=T_{fg}^\forall$ and hence holds in
$\mathbb F_2*\langle z\rangle$. This is a contradiction, since $z$ is
a square root of $z^2$.
\end{proof}

Finally, by \cite{MR1043446}, every nonabelian limit group is universally
equivalent to a nonabelian free group. Hence the main theorem also yields:

\begin{corollary}
The first-order theory of any nonabelian limit group does not admit a model companion.
\end{corollary}

\end{document}